\documentclass[11pt]{amsart}
 \usepackage{amsmath,amssymb,enumerate}
 \usepackage[all]{xy}
 \usepackage{xcolor} 
 
 \newtheorem{pro}{Proposition}[section]
 \newtheorem{cor}[pro]{Corollary}
 \newtheorem{lem}[pro]{Lemma}
 \newtheorem{rem}[pro]{Remark}
 \newtheorem{exa}[pro]{Example}
 
  \newtheorem{defi}[pro]{Definition}
 \newtheorem{theo}[pro]{Theorem}

 \newcommand{\R}{\mathbb R}
 
 \newcommand{\C}{\mathbb C}
 \newcommand{\N}{\mathbb N}

 \newcommand \mrm{\mathrm}

 \numberwithin{equation}{section}

\usepackage{hyperref}

\begin{document}
 \title[  Applications of Logarithmic Potentials]{ Some applications of Projective Logarithmic Potentials}
\setcounter{tocdepth}{1}
\author{Sa\"{\i}d Asserda and Fatima Zahra Assila} 
\address{Ibn Tofail University \\ Faculty of Sciences  \\ PO 242 \\
Kenitra \\
Marocco}
\email{said.asserda@uit.ac.ma, fatima.zahra.assila@uit.ac.ma}
 \date{\today}
\begin{abstract}
  We continue the study in \cite{As18, AAZ18} by giving a multitude of applications of projective logarithmic potentials. First we  introduce the notions of projective logarithmic energy and capacity associated to projective kernel that was  introduced and studied in \cite{As18, AAZ18}.  We compare quantitatively the projective logarithmic capacity with the complex Monge-Amp\`ere capacity on $\mathbb P^n$ and we deduce that the set of zero logarithmic capacity is of Monge-Amp\`ere capacity zero. Further, we define transfinite diameter of a compact set and we show that it coincides with logarithmic capacity. Finally we deduce that there is an analogous of  classical Evans's theorem that  for any compact set $K$ of zero projective logarithmic capacity shows the existence of Probability measure whose potential admits $K$ as polar set.
 \end{abstract}
\maketitle
\tableofcontents
\section{Introduction}
In Potential Theory different notions of energies (depending on the problem under study) associated to Borel measures, were introduced with respect to theirs  potentials (\cite{C65, Ran95, Pa04}). They allow to define capacities as a set functions which characterize small sets. On the other hand, the kernel allows to define the transfinite diameter and chebyshev constant notions. These help to give another approach to capacity which has revealed so useful in the local theory.\\
  In the complex plane (\cite{Ran95}), capacities enjoy several natural properties. This is due to the fact that logarithmic kernel satisfies the maximum principle. The importance of the maximum principle lies in the fact that from local assumptions it derives global conclusions. Such results are usually very powerful, and the maximum principle is no exception. Carleson \cite{C65} gave  an exceptional class of examples satisfying the maximum principle in  the euclidean space $\mathbb{R}^d,\;\; d\geq 3 $.  The situation is different for the logarithmic potential in higher complex dimension, since there is no maximum principle.\\  
  The aim of this paper  is to investigate  which properties of classical potentials remain true in  the context of complex projective space. More precisely,  the projective logarithmic  potential $G_\mu$ which is defined earlier in \cite{As18, AAZ18} enjoys several interesting properties like in the local setting. For applicatons: we  study the projective logarithmic energy of a Probability measure on the complex projective space $\mathbb P^n$ that allows to define the projective logarithmic  capacity $\kappa$ and $\kappa-$polar sets.\\ 
   We also introduce and study the projective logarithmic transfinite diameter, the projective chebyshev constant and we show that there is a relationship between them and capacity $\kappa$. We conclude by proving our main result which  is  a strong generalization of classical Evans's theorem. \\ 
 We briefly describe the main results and ideas of this paper:  let $\mu$  be a Probability measure on the complex projective space $\mathbb P^n $. Then its projective logarithmic capacity is defined on $Prob(\mathbb P^n)$ as follows:
  $$
 \kappa (E) :=\sup \{ e^{- I(\mu)}; \; \mu \in \text{Prob}  (\mathbb P^n), \mrm{ Supp} \mu \subset E\}
 $$
 where 
 \begin{eqnarray*}
I(\mu)&:=&- \int_{\mathbb P^n} G_\mu(\zeta)d\mu(\zeta)\\
&:=&  - \int_{\mathbb P^n}   \int_{\mathbb P^n} \log \left (\frac{\vert \zeta \wedge \eta\vert}{\vert \zeta\vert \vert \eta\vert}\right)  d \mu (\zeta) \ d \mu (\eta), 
 \end{eqnarray*}
 is the projective logarithmic energy.\\
 Using results and tools in \cite{As18, AAZ18}, we show that $\kappa-$polars sets (i.e a set has zero projective logarithmic capacity), are pluripolars with a precise quantitative estimate of  the Monge-Amp\`ere Capacity in term  of the  logarithmic capacity :\\
{\bf Theorem A}. {\it Let $E \subset \mathbb P^n$ be a Borel set. Then 
 $$
 \mathrm{Cap}_{FS}^\ast (E) \leq \frac{\sqrt{s_n}}{\sqrt{- \log \kappa (E)}},
 $$
 where $s_n$ is a constant of lemma \ref{lem: bounded}.
 In particular, every Borel $\kappa$-polar set is pluripolar.}\\ 
  If $ \mu \in \text{Prob}  (\mathbb P^n)$ be a measure with 
$I(\mu)< +\infty$, then the Borel set $$E=\{ \zeta\in \mathbb P^n; \; G_{\mu}(\zeta)=-\infty\}$$ is $\kappa-$polar. The reverse of this result constitutes the main theorem of this paper, which generalizes and improves Evans's theorem in local  setting. \\
  The second fundamental objects that we are going to prove in this paper can be stated as follows :\\
   {\bf Theorem B}. {\it Let $E \subset \mathbb{P}^n$ be a closed set such that $\kappa(E)=0$. Then there exists a Probability measure  $\mu$ in  $\mathbb{P}^n$ with support in   $E$ such that
  $$
  E =\{\zeta \in \mathbb P^n /\;\; {G}_\mu (\zeta) = -\infty\}.
  $$
  In particular $E$ is a complete pluripolar set of $\mathbb P^n$.}
\section{Semi-continuity of the potentials} 
  We study the upper semi-continuity property of the operator $\mu \longmapsto \mathbb{G}_{\mu}$. We use the weak-topology on the compact space $\text{Prob} (\mathbb P^n)$ of Probability measures on $\mathbb P^n$. We recall that this space is compact.
  \begin{pro}\label{pro: semi-cont}
 Let $(\mu_j)$ be a sequence converging to $ \mu$ in  $\text{Prob} (\mathbb P^n)$. Then
\begin{equation} \label{eq:scont1}
 \limsup_{j \to + \infty} G_{\mu_j} (\zeta) \leq G_{\mu} (\zeta), \, \, \, \forall \zeta \in \mathbb P^n,
\end{equation}
and 
\begin{equation} \label{eq:scont2}
  \limsup_{j \to + \infty}  \int_{\mathbb P^n} G_{\mu_j} (\zeta) d \mu_j (\zeta) \leq \int_{\mathbb P^n} G_{\mu} (\zeta) d \mu (\zeta). 
\end{equation}

 \end{pro}
 \begin{proof} Observe that by definition, the integrand is $\leq 0$. Hence by upper semi-continuity and Fatou's lemma, we have

 $$
 \limsup_j G_{\mu_j} (\zeta) \leq \mathbb{G}_{\mu} (\zeta),
 $$
 which proves the inequality (\ref{eq:scont1}). 
Integrating and applying Fatou's again lemma yields the inequality(\ref{eq:scont2}).
 \end{proof}
 
The potential  $G_{\mu}$ enjoys several other interesting properties, we now prove one of these: the continuity principle.

\begin{theo}[Continuity Principle]{\it Let $\mu$ be a Probability measure on $\mathbb{P}^{n}$ and let $ K := \text{Supp} \, \mu \subset \mathbb{P}^n$.\\

1. If $\zeta_{0}\in K$, then $$\liminf_{\zeta \to \zeta_{0}}G_{\mu}(\zeta)=\liminf\limits_{\underset{\zeta^{\prime}\in K}{\zeta^{\prime} \to \zeta_{0}}}G_{\mu}(\zeta^\prime).$$

2. If further $$\lim\limits_{\underset{\zeta^{\prime}\in K}{\zeta^{\prime} \to \zeta_{0}}}G_{\mu}(\zeta^{\prime})=G_{\mu}(\zeta_0),$$ then $$\lim_{\zeta \to \zeta_{0}}G_{\mu}(\zeta)=G_{\mu}(\zeta_0).$$ 
}\end{theo}

We recall that  the potential $G_{\mu}$ can expressed in terms of the geodesic distance $d$   
( see \cite{As18})  on the k\"{a}hler manifold $(\mathbb{P}^n,\omega_{FS})$, namely we have 
 
 $$G_{\mu} (\zeta)=\int_{\mathbb{P}^n}\log\sin\frac{d(\zeta,\eta)}{\sqrt{2}}d\mu(\eta),\;\;\;\;  \;\; \zeta\in \mathbb{P}^n.$$

\begin{proof}
1. If $G_{\mu}(\zeta_0)=-\infty$, then we have  by using upper semi-continuity that
\begin{eqnarray*}
\lim_{\zeta \to \zeta_{0}}G_{\mu}(\zeta)&\leq& \limsup_{\zeta \to \zeta_{0}}G_{\mu}(\zeta)\\
&\leq& G_{\mu}(\zeta_0)\\
&=& -\infty, 
\end{eqnarray*}
this implies that $$\lim_{\zeta \to \zeta_{0}}G_{\mu}(\zeta)=-\infty,$$ also 

\begin{eqnarray*}
\lim_{\zeta^\prime \to \zeta_{0}}G_{\mu}(\zeta^\prime)&\leq& \limsup_{\zeta^\prime \to \zeta_{0}}G_{\mu}(\zeta^\prime)\\
&\leq& \limsup_{\zeta \to \zeta_{0}}G_{\mu}(\zeta)\\
&\leq& G_{\mu}(\zeta_0)\\
&=& -\infty, 
\end{eqnarray*}
this implies that $$\lim_{\zeta^\prime \to \zeta_{0}}G_{\mu}(\zeta^\prime)=-\infty,$$

 and we have $$\liminf_{\zeta \to \zeta_{0}}G_{\mu}(\zeta)=\liminf\limits_{\underset{\zeta^{\prime}\in K}{\zeta^{\prime} \to \zeta_{0}}}G_{\mu}(\zeta^\prime).$$
Thus we can suppose that $G_{\mu}(\zeta_0)>-\infty$. Then we have $\mu(\{\zeta_0\})=0$. Given  $\epsilon>0$, we can thus find $0<r<<1$, suth that $\mu(B(\zeta_0,r))<\epsilon$. 

For a given $ \zeta \in \mathbb{P}^n $. Choose $ \zeta^\prime\in K $ such that,  $$ \sin\Big(\frac{ d(\zeta^\prime,\zeta)}{\sqrt{2}}\Big) =\inf_{\eta\in \mathbb{P}^n} \sin\Big(\frac{ d(\eta,\zeta)}{\sqrt{2}}\Big).$$ 

Since $ 0\leq d(\zeta,\eta)\leq \frac{\pi}{\sqrt{2}}$,  for all $ \eta\in K$.

\begin{eqnarray*}
\frac{\sin \Big(\frac{ d(\zeta^\prime,\eta)}{\sqrt{2}}\Big)}{\sin\Big(\frac{d(\zeta,\eta)}{\sqrt{2}}\Big)}&\leq& \frac{\sin\Big(\frac{d(\zeta^\prime,\zeta)}{\sqrt{2}}+\frac{ d(\zeta,\eta)}{\sqrt{2}}\Big)}{\sin\Big(\frac{ d(\zeta,\eta)}{\sqrt{2}}\Big)}\\
&\leq&\frac{\sin\Big(\frac{d(\zeta^\prime,\zeta)}{\sqrt{2}}\Big)+\sin\Big(\frac{ d(\zeta,\eta)}{\sqrt{2}}\Big)}{\sin\Big(\frac{ d(\zeta,\eta)}{\sqrt{2}}\Big)}\\
&\leq& 2.
\end{eqnarray*}

Therefore, 
\begin{eqnarray*}
G_{\mu}(\zeta)-G_{\mu}(\zeta^\prime)&=& \int_{\mathbb{P}^n}\log\Big(\frac{\sin \Big(\frac{ d(\zeta,\eta)}{\sqrt{2}}\Big)}{\sin\Big(\frac{d(\zeta^\prime,\eta)}{\sqrt{2}}\Big)}\Big)d\mu(\eta)\\
&=& \int_{K}\log\Big(\frac{\sin \Big(\frac{ d(\zeta,\eta)}{\sqrt{2}}\Big)}{\sin\Big(\frac{d(\zeta^\prime,\eta)}{\sqrt{2}}\Big)}\Big)d\mu(\eta)\\
&=&\int_{B(\zeta_0,r)\cap K}\log\Big(\frac{\sin \Big(\frac{ d(\zeta,\eta)}{\sqrt{2}}\Big)}{\sin\Big(\frac{d(\zeta^\prime,\eta)}{\sqrt{2}}\Big)}\Big)d\mu(\eta) +\\
&{}& \int_{K\setminus B(\zeta_0,r)}\log\Big(\frac{\sin \Big(\frac{ d(\zeta,\eta)}{\sqrt{2}}\Big)}{\sin\Big(\frac{d(\zeta^\prime,\eta)}{\sqrt{2}}\Big)}\Big)d\mu(\eta)\\
&\geq& -\epsilon\log(2)+\int_{K\setminus B(\zeta_0,r)}\log\Big(\frac{\sin \Big(\frac{ d(\zeta,\eta)}{\sqrt{2}}\Big)}{\sin\Big(\frac{d(\zeta^\prime,\eta)}{\sqrt{2}}\Big)}\Big)d\mu(\eta).
\end{eqnarray*}

Observe that for  $\zeta$ close to  $\zeta_0$ in $\mathbb{P}^n$, we can choose  the corresponding projection $\zeta^\prime$ close to $ \zeta_0$ in $K$, and hence 
$$\liminf_{\zeta \to \zeta_{0}}G_{\mu}(\zeta)\geq  \liminf\limits_{\underset{\zeta^{\prime}\in K}{\zeta^{\prime} \to \zeta_{0}}}G_{\mu}(\zeta^\prime)- \epsilon\log(2).$$
Since $\epsilon$ is arbitrary, the result follows.\\

2. If $G_\mu$ satisfies the promise of 2., then by part 1 of theorem we have,
$$
\liminf_{\zeta \to \zeta_{0}}G_{\mu}(\zeta)=G_{\mu}(\zeta_0).
$$ 
As $G_{\mu}$ is upper semicontinuous,  
$$
\limsup_{\zeta \to \zeta_{0}}G_{\mu}(\zeta)\leq G_{\mu}(\zeta_0).
$$
Combining these two results yields, 

 $$
\lim_{\zeta \to \zeta_{0}}G_{\mu}(\zeta)=G_{\mu}(\zeta_0),
$$
 as required.
\end{proof}
 \section{Projective logarithmic energy} 
  Let us define {\it the mutual (projective logarithmic) energy} of two Borel  measures $\mu_{1},\mu_{2} $ in $\text{Prob}  (\mathbb P^n)$ by
 \begin{eqnarray}
  I_G (\mu_{1},\mu_{2})&:=& \langle\mu_1,\mu_2\rangle \\
   & := &   - \int_{\mathbb P^n} {G}_{\mu_{1}} (\zeta) d \mu_{2} (\zeta)  \nonumber\\
  & =&   - \int_{\mathbb P^n}   \int_{\mathbb P^n} \log \left (\frac{\vert \zeta \wedge \eta\vert}{\vert \zeta\vert \vert \eta\vert}\right)  d \mu_{1} (\eta) \ d \mu_{2} (\zeta),
 \end{eqnarray}
 which is either a non-negative number or $+\infty$.\\
In case $ \mu_{1}=\mu_{2}=\mu$,  we call $ I_G (\mu)    := I_G (\mu,\mu)$ {\it the  projective logarithmic energy of $\mu $} i.e.
 \begin{eqnarray}
  I_G (\mu)   & := &  - \int_{\mathbb P^n} {G}_{\mu} (\zeta) d \mu (\zeta)  \nonumber\\
  & =&   - \int_{\mathbb P^n}   \int_{\mathbb P^n} \log \left (\frac{\vert \zeta \wedge \eta\vert}{\vert \zeta\vert \vert \eta\vert}\right)  d \mu (\eta) \ d \mu (\zeta) \geq 0.
 \end{eqnarray}
 It is clear that $I (\mu) \in [0 , + \infty]$. \\
It follows from the upper semi-continuity of the potentials that the energy is  upper semi-continuous i.e. 
\begin{pro}\label{pro:lower}
The operator $I : \mathcal M^+ (\mathbb P^n) \longrightarrow \R^+ \cup \{+ \infty\}$ is lower semi-continuous: If $(\mu_j)$ be a sequence weakly converging to $ \mu$ in  $\text{Prob} (\mathbb P^n)$. Then
$$
 I_G (\mu) \leq \liminf_{j \to + \infty} I_G (\mu_j).
 $$
\end{pro}
\begin{proof}
Is an immediate consequence of the  semi-continuity property of $G$  proposition \ref{pro: semi-cont} and Fatou's lemma.
\end{proof}
\begin{defi}
We say that a positive Borel measure $\mu$ on $\mathbb P^n$ is a finite $G$-energy if $0\leq I_{G}(\mu)<+\infty$.
\end{defi}
In the sequel,   ${\mathcal {M}_G}^+ (\mathbb P^n)$ will denote the space of  positives Borel measures of finite energy.\\
 We now establish an important lemma which expresses
 polarization identity for the projective logarithmic energy:
\begin{lem}\label{lem:Fonda}
Let  $\mu,\nu \in {\mathcal {M}_G}^+ (\mathbb P^n)$ be a positive Borel measures of finite energy. Then   $0 \leq I_G (\mu,\nu) < + \infty$ and we have the following polarization identity:
 \begin{equation} \label{eq:polarisation}
  I_G (\mu + \nu) = I_G (\mu) + I_G (\nu) + 2 I_G(\mu ,\nu).
 \end{equation}
 In particular $\mu + \nu  \in {\mathcal{M}_G}^+ (\mathbb P^n)$.
\end{lem}
\begin{proof}
The lemma would be an easy exercice of linear algebra if we know how to show that the kernel $G$ is positive definite i.e., the mutual energy associated $ I_G (\mu, \nu) $ is symmetric bilinear form, positive definite in the real vector space $\mathcal {M}^0_G (\mathbb P^n) $ of the signed Borel measures in $\mathbb P^n$ of zero total mass and finite energy. It is not obvious and constitutes an interesting question which is still open.\\
We will localize the problem and use the classical results of potential theory in $\R^{2 n}$. Recall first some necessary notations from (\cite{La72}, Chapter 1).\\
 Let $Q (z,w)$  be a kernel on $\C^n \times \C^n$ i.e. a locally  bounded (from) above  Borel function with values in $\R \cup \{- \infty \}$. We define the energy associated to $Q$ on $\mathcal M (\mathbb C^n) $ the set of signed Borel measures with compact support by the following formula: If $\mu, \nu \in \mathcal M (\mathbb C^n) $ are  compactly supported,  we put 
$$
 I_Q (\mu,\nu) := - \int_{\C^n } \int_{\C^n} Q (z,w) d \mu (z) d \nu (w).
$$
If $\mu = \nu$ we put $I_Q (\mu) := I_Q (\mu,\mu)$. By hypothesis, $I_Q (\mu,\nu)  \in \R \cup \{+ \infty\}$. We will say that $\mu$ is of finite energy (with respect to $Q$) if  $I_Q (\vert \mu \vert) < \infty.$\\
Note that $\mathcal M_{Q} (\mathbb C^n) $ is the set of signed Borel measures with compact support of finite energy and  $\mathcal M^0_{Q} (\mathbb C^n) $ the subset formed of which are zero total mass.\\
It is well known that the Riesz kernel $R_{\alpha}$ defined on $\C^n \simeq \R^{2 n}$ by
 $$
 K_{\alpha} (z,w) =R_{\alpha} (z,w) := - \vert z - w \vert^{\alpha - 2 n}, \, \, 0 < \alpha < 2 n,
 $$ 
 is a positive definite in the sens that in the mutual energy $I_{K_{\alpha}}$ associated to Riesz's kernel is
 scalar product on the space $\mathcal M_{K_{\alpha}} (\mathbb C^n) $ of signed Borel measures with compact support and finite energy.\\
 It follows that by letting $\alpha \to 2 n$ that the mutual energy $I_L$ associated to the logarithmic potential $L$ defined by
$$
 L (z,w) :=  \log \vert z- w \vert^2, \, \, (z,w) \in \C^n \times \C^n,
$$
 possess the same property  provided  to reduce on the space   of signed Borel measures with compact support, of zero total mass and of finite energy (see \cite{La72}, page 50).\\
More precisely, the mutual energy $I_L$ associated to the kernel $L$ on $\mathbb C^n$ defined for two signed measures 
$\mu, \nu $ of finite  energy by
$$
 I_L (\mu,\nu) :=  - \int_{\C^n \times \C^n} L (z,w) d \mu (z) d \nu (w),
$$ 
is positive definite on the vector space ${\mathcal M}_L^0 (\C^n)$   of the signed measures with zero mass on $\C^n \simeq \R^{2 n}$ with compact support of finite energy (see \cite{La72}).\\ 
It follows that: If $\mu$ and $\nu$ are the Probability measures with compact supports, then the signed measure (i.e. any difference of two non-negative Borel
measures  at least one of which is finite)  $\mu - \nu$ with compact support of zero mass total of finite energy, we have then 
$$
  0 \leq I_L (\mu - \nu) = I_L (\mu) + I_L (\nu) - 2 I_L (\mu,\nu).
 $$
Hence we deduce the  following fundamental inequality: If 
$\mu$ and $\nu$ are two Probability measures with compact supports in $\C^n$, we have 
\begin{equation} \label{eq:InegFond1}
 2 I_L (\mu,\nu) \leq I_L (\mu) + I_L (\nu). 
 \end{equation} 
 Consider the normalized projective logarithmic kernel studied recently in (\cite{As18, AAZ18}) and defined   for 
 $(z,w) \in \C^n \times \C^n$ as follows
  $$
  \mathcal N (z,w) :=  (1 \slash 2) \log \frac{\vert z - w\vert^2 + \vert z \wedge w\vert^2}{(1 + \vert z\vert^2) (1 + \vert w\vert^2)}.
 $$
 It is not known whether the kernel $\mathcal N$ possess the property (\ref {eq:InegFond1}).\\
 Observe that for $(z,w) \in \C^n \times \C^n$, the function 
 $$
 S (z,w) :=   \mathcal N (z,w) - L (z,w) = (1 \slash 2) \log \frac{\vert z - w\vert^2 + \vert z \wedge w\vert^2}{\vert z - w\vert^2 (1 + \vert z\vert^2) (1 + \vert w\vert^2)},
 $$  
 is a locally bounded negative function in $\C^n \times \C^n$. It follows that the Borel measures with compact support of finite energy are the same for the two kernel $L$ and  $\mathcal  N$ and that
 $$
I_{\mathcal N}  (\mu,\nu) = I_L (\mu, \nu) + I_S (\mu,\nu),
$$  
where $I_S (\mu,\nu)$ is a bounded function whose bound  depend only upon  the bound of  $S$ on the  support  of $\mu \otimes \nu$.
It follows that if $\mu$ and $\nu$ are two Probability measures with compact supports in  $\C^n$ of finite energy, $0 \leq I_L (\mu - \nu)  < + \infty $ 
and we have
\begin{eqnarray*}
   I_{ \mathcal N} (\mu - \nu) & = &  I_L (\mu - \nu) + I_S (\mu - \nu) \\
   & = &  I_L (\mu) + I_L (\nu) - 2 I_L (\mu,\nu) + I_S (\mu - \nu) \\
   & =&  I_{\mathcal N} (\mu) + I_{\mathcal N} (\nu) - 2 I_{\mathcal N} (\mu,\nu).
\end{eqnarray*}
Therefore we deduce that if $\mu$ and $\nu$ are two Probability measures with compact supports in $\C^n$, of finite energy, we have
$$
I_S (\mu - \nu)  \leq  I_{\mathcal N} (\mu) + I_{ N} (\nu) - 2 I_{ N} (\mu,\nu) . 
$$
In particular $I_{\mathcal N} (\mu,\nu) < + \infty.$\\
The first assertion of the lemma will follow namely, if $\mu, \nu \in \mathcal M_G (\mathbb P^n)$ are two  Probability measures in $\mathbb P^n$ of  finite energy, then 
$$
0 \leq I_G (\mu,\nu) < + \infty.
$$
Recall that $\mathbb P^n = \cup_{0 \leq j \leq n} \, \, \mathcal U_j$, where \, \, $ \mathcal U_j := \{ \zeta \in \mathbb P^n ; \zeta_j \neq 0\}$.\\
In each open $\mathcal U_j$, we have  the coordinates system  $z = \phi_j (\zeta) $ so that in  $\mathcal U_j \times \mathcal U_j$, we have
$$
 G(\zeta,\eta) = {\mathcal N} (z,w), \, \,  \mathrm{with} \, \, \,  z := \phi_j (\zeta), w := \phi_j (\eta).
 $$
1) Suppose now that $\mu$,  $\nu$ are supported in the same   domain of chart  $\mathcal U_j $. Then
$$
I_G (\mu,\nu) = I_{\mathcal N} (\tilde \mu,\tilde \nu), 
$$
 where $\tilde \mu = (\phi^j)_* \mu$ and $\tilde \nu = (\phi^j)_* \nu$  are two prabability measures with compact support in  $\C^n$ of finite energy with respect to $\mathcal N$.  Which proves  $0 \leq I_G (\mu,\nu) < + \infty$ according to the foregoing.\\
2) Suppose $\mu$ is supported on  $\mathcal U_j$ and $\nu$ is supported on $\mathcal U_k$ with $j \neq k$.  Consider the transformation $\psi_{j,k} : \mathcal U_k \longrightarrow \mathcal U_j$ which consists to switch the coordinates  $\zeta_j$ and  $\zeta_k$. Then, putting  $\hat \nu := (\psi_{j,k})_* \nu$, here we obtain a Probability measure with compact measure in  $\mathcal U_j$ such that 
$$
I_G (\mu,\nu) := I_G (\mu,\hat \nu) < + \infty,
$$
according to the first step.\\
3) If $\mu$  and  $\nu$ are two Probability measures in  $\mathbb P^n$ of finite energy, we can decompose  
$\mu = \sum_{j = 0}^n \alpha_j \mu_j$ and $\nu = \sum_{k = 0}^n \beta_k \nu_k$ into convex combination Probability measures with finite energy such that  $\mu_j $ and $\nu_j$ be supported in $\mathcal U_j$ for $j = 0, \cdots, n$.\\
It follows by bilinearity that
$$
 0 \leq I_G (\mu,\nu) = \sum_{j,k = 0}^n \alpha_j \beta_k I_G (\mu_j,\nu_k) < + \infty.
$$
The polarization identity is a consequence of the bilinearity and the symmetry of $I_G (\mu,\nu)$ in the space $\mathcal {M}_G (\mathbb P^n) $. Which finishes the proof of the lemma.  
 \end{proof}
 In the sequel we write  $I = I_G$ to simplify notations.
 \subsection{ Calculus the energy of some measures}
  In this subsection we will discuss about energy of two extreme examples of measures : Dirac measure and Fubini-Study measure.
  \begin{exa}
We first consider, Dirac measure  $\mu:=\delta_{a}$ at the point  $a\in \mathbb P^{n}$.  We have
 \begin{eqnarray*}
 I(\delta_{a})&:=&-\int_{\mathbb P^{n}}\int_{\mathbb P^{n}}\log \left (\frac{\vert \zeta \wedge \eta\vert}{\vert \zeta\vert \vert \eta\vert}\right)  d \delta_{a} (\zeta) \ d \delta_{a} (\eta)\\
 &=& -\log \left (\frac{\vert a \wedge a \vert}{\vert a \vert \vert a\vert}\right)\\
 &=&+\infty.
 \end{eqnarray*}
 More generally if  $a, b \in  \mathbb P^{n}$ such that $a \neq b$, we have 
 $$
 I(\delta_a , \delta_b) = -\log \left (\frac{\vert a \wedge b \vert}{\vert a \vert \vert b\vert}\right) < + \infty.
 $$
 Observe that polarization identity (\ref{eq:polarisation}) is valid for all Borel measures in $\mathbb P^n$. It follows that $I (\delta_a + \delta_b) = + \infty$.
 \end{exa}
 \begin{exa} \label{ex:FS} We now consider  Fubini$-$Study  measure 
  $\mu_{FS}:=\frac{ \omega^{n}_{FS}}{n!}=dV_{FS}$, we have  $0\leq d(\zeta,\eta)\leq \frac{\pi}{\sqrt{2}}$ for all $\zeta,\eta\in \mathbb{P}^{n}$.

  \begin{eqnarray*} 
 I(\mu_{FS}):&=&-\int_{\mathbb P^n}\Big(\int_{\mathbb P^n}\log\sin \frac{d(\zeta,\eta)}{\sqrt{2}}dV_{FS}(\eta)\Big)dV_{FS}(\zeta) \\
&=& -c_{n} \int_{0}^{\pi\over\sqrt{2}}\log\sin\bigl({r\over\sqrt{2}}\bigr)A(r)dr\\
 &=&-c_{n} \int_{0}^{\pi\over\sqrt{2}}\log\Bigl(\sin\bigl({r\over\sqrt{2}}\bigr)\Bigr)\sin^{2n-2}\bigl({r\over\sqrt{2}}\bigr)
 \sin(\sqrt{2}r)dr\\
  &=& -2\sqrt{2}c_{n} \int_{0}^{\pi\over 2}(\log\sin(t))\sin^{2n-1}(t)\cos(t)dt\\
 &=&-2\sqrt{2}c_{n} \int_{0}^{1}u^{2n-1}\log u du\\
 &=&2\sqrt{2}c_{n} \frac{1}{(2n)^2}\\
  &<&+\infty.
 \end{eqnarray*} 
where  $A(r):=c_{n}\sin^{2n-2}(r/\sqrt{2})\sin(\sqrt{2}r)$ is the area of the sphere of center $\eta$ and radius  $r$ of  $\mathbb{P}^{n}$ and $c_{n}$ a positive constant. 
\end{exa} 
 \subsection*{$\kappa-$Polar Set:} The $\kappa-$polar  set play the role of negligible sets in projective logarithmic potentials,  much as sets of measure zero do in potentials theory and measure theory. It follows from the projective logarithmic energy the following definition.
   \begin{defi}
 A subset $E$ of $\mathbb{P}^{n}$ is called $\kappa-$polar if $I(\mu)=+\infty$ for every Borel measure $\mu \in  \mathrm{Prob}(\mathbb P^n)$ for which $\mathrm{Supp} \mu \, \, \subset E$.
 \end{defi}
 \begin{exa}
 It follows from previous example that every finite or countable set of $\mathbb P^n$ is $\kappa$-polar.
  \end{exa} 
 The following proposition of [\cite{Ran95}] generalized to $\mathbb P^n$  shows in particular that the measures of finite energy do not charge $\kappa-$polar sets.
 \begin{pro} \label{pro: finiteEnergie}
 Let $\mu $ be a finite Borel measure in  $Prob(\mathbb P^n)$, and  suppose that $I(\mu)<+\infty$. Then $\mu(E)=0$ for every Borel $\kappa-$polar set $E$ of $\mathbb P^n$.
\end{pro}  
\begin{proof}
  Suppose $E $ be a Borel set of $\mathbb P^n$  such that  $\mu(E)>0$. There exists a compact subset $K\subset E$ such that $\mu(K)>0$. Set $\nu=1_{K}\frac{\mu}{\mu(K)}$. Then $\nu$ is a Probability measure whose support in $ E$ and its energy checks:
 \begin{eqnarray}
   I(\nu)
  &=& -  (\mu (K))^{- 2} \int_{K}   \int_{K} \log \left (\frac{\vert \zeta \wedge \eta\vert}{\vert \zeta\vert \vert \eta\vert}\right)  d \mu (\zeta) \ d \mu (\eta)\nonumber\\
  &=&- (\mu (K))^{- 2} \int_{ K} \int_{ K} \log \sin\Big(\frac{d(\zeta,\eta)}{\sqrt{2}}\Big)d\mu(\eta)d\mu(\zeta)  \nonumber\\
  &\leq &- (\mu (K))^{- 2} \int_{\mathbb P^n } \int_{\mathbb P^n} \log \sin\Big(\frac{d(\zeta,\eta)}{\sqrt{2}}\Big)d\mu(\eta)d\mu(\zeta)\nonumber\\
  &=& (\mu (K))^{- 2}  I(\mu) < + \infty\nonumber.
  \end{eqnarray}
  It follows that $E$ is not $\kappa-$polar. 
  \end{proof}
 \section{Projective logarithmic  capacity and proof of Theorem A} 
 Recall the definition of the projective logarithmic energy of $\mu $ in $\text{Prob}  (\mathbb P^n)$ from the previous section, 
 \begin{eqnarray*}
  I (\mu) & := &  - \int_{\mathbb P^n} \mathbb{G}_{\mu} (\zeta) d \mu (\zeta) \\
  & =&   - \int_{\mathbb P^n}   \int_{\mathbb P^n} \log \left (\frac{\vert \zeta \wedge \eta\vert}{\vert \zeta\vert \vert \eta\vert}\right)  d \mu (\zeta) \ d \mu (\eta) \in [0,+\infty].
 \end{eqnarray*}
 We define for a subset $E \subset \mathbb P^n$, the Robin constant  
 $$
 \gamma (E) := \inf \{I (\mu) ; \mu \in \text{Prob}  (\mathbb P^n), \mrm{ Supp} \mu \subset E \},
  $$
  and its projective logarithmic capacity as 
  $$
 \kappa (E) := e^{- \gamma (E)}.
  $$
  Observe that $\kappa (E) > 0$ if and only if there exists $\mu \in \text{Prob}  (\mathbb P^n), Supp \mu \subset E$ such that $I (\mu) < + \infty$.\\
  By duality, we obtain the following:
  \begin{pro} \label{pro:dualité} For all Borel $E \subset \mathbb P^n$, we have
  \begin{equation} \label{eq:dualite}
  \frac{1}{ \sqrt{- \log \kappa  (E)} }= \sup \{\nu (E) ; \nu \in \mathcal M^+ (\mathbb{P}^n), \mrm{ Supp} \nu \subset E,   I (\nu) \leq 1 \}.
  \end{equation}
  In particular,  $E$ is $\kappa$-polar if and only if   $\nu (E) = 0$ for all Borel measure  $ \nu \in \mathcal M^+ (\mathbb{P}^n)$ such that $I (\nu) < + \infty$. 
  \end{pro}
  \begin{proof}  By definition, we have 
  $$
  \sqrt{- \log \kappa (E)} :=\inf  \{ \sqrt{ I (\mu) } ; \mu \in \text{Prob}  (\mathbb P^n), \mrm{ Supp} \mu \subset E\}.
  $$
We denote by $\sigma (E)$ the second  member of the equality (\ref{eq:dualite}).  Let $\mu \in \text{Prob}  (\mathbb P^n), \mrm{ Supp} \mu \subset E $ and $I (\mu) < +\infty$. Then for all $\varepsilon > 0$ small enough, $\nu := \mu \slash \sqrt{I (\mu) + \varepsilon} \in \mathcal M^+ (\mathbb{P}^n),$ and checks $ I (\nu) \leq 1$. We infer that $\nu (E) \leq \sigma (E)$. Since $\mu (E) = 1$, it follows that 
 $$ 1 \slash \sqrt{I (\mu) + \varepsilon} =  \mu (E)\slash \sqrt{I (\mu) + \varepsilon} \leq \sigma (E).$$
 Taking the upper bound on $\mu$ and tends $\varepsilon $ to $0$,  yields the first inequality in  (\ref{eq:dualite}).
  Conversely, let $\nu \in \mathcal M^+ (\mathbb{P}^n),  I (\nu) \leq 1 $ such that $\nu (E) > 0$. Then $\mu := {\bf 1}_E \nu \slash \nu (E)$ is a Probability measure with compact support in  $E$ such that  $I (\mu) \leq I (\nu) \slash (\nu (E))^2 \leq 1 \slash (\nu (E))^2$. This ensures that $\gamma (E) \leq  1 \slash (\nu (E))^2$.
  Therefore $\nu (E) \leq 1 \slash \sqrt{\gamma (E)}$.
Taking the supremum over  $\nu$, we obtain the second required inequality  $\geq$ in (\ref{eq:dualite}).
  \end{proof}
   As a consequence we state the following results:
 \begin{pro} The set function  $E \longmapsto  \frac{1}{ \sqrt{- \log \kappa  (E)} }$ is subadditive on the $\sigma-$Algebra of Borel subsets   $\mathcal B (\mathbb P^n)$  of $\mathbb P^n$ i.e., if $(E_j)$ is a sequence of Borel sets of $\mathbb P^n$ and $E := \cup_j E_j$, then
  $$
  \frac{1}{\sqrt{- \log \kappa  (E)} } \leq \sum_{j = 0}^{+ \infty}  \frac{1}{ \sqrt{- \log \kappa (E_j)}}.
 $$
 \end{pro}
  \begin{proof} A consequence of the formula of duality and by subadditivity of Borel measures on $\mathbb P^n$.
   \end{proof}
 \begin{cor} For all Borel $E \subset \mathbb P^n$, we have
 $$
 \mathrm{Vol}_{FS} (E) \leq \frac{a_n}{\sqrt{- \log \kappa  (E)}},
 $$
 where $a_n > 0$ is an absolute constant.
\end{cor}
 \begin{proof} It follows from  Example  \ref{ex:FS} that  there exists a constant  $b_n > 0$ such that $ I (\mathrm{Vol}) \leq b_n$. It follows that $\nu := \mathrm{Vol}_{FS} \slash \sqrt{b_n} \in  \mathcal M^+ (\mathbb{P}^n)$ and checks $  I (\nu) \leq 1$. It follows from proposition \ref{pro:dualité} that $\nu (E) \leq \frac{1}{ \sqrt{- \log \kappa  (E)} }$, which yields the desired inequality with $a_n := \sqrt{b_n}$.
 \end{proof}
   \begin{lem}\label{lem: bounded}
There exists a constant $s_n > 0$ such that if  $u: \mathbb{P}^{n}\longrightarrow [-\infty, +\infty[ $ is  $\omega$-plurisubharmonic on $\mathbb{P}^{n}$  suth that $- 1 \leq u \leq 0$, the measure $\mu =(\omega+dd^{c}u)^{n}$ verifies   $I(\mu_u) \leq s_n$.
 \end{lem}
\begin{proof}  Let $\mu =(\omega+dd^{c}u)^{n} $ with  $u\in PSH(\mathbb{P}^{n},\omega)\cap L^{\infty}(\mathbb{P}^{n})$. We have then 
\begin{eqnarray*}
I(\mu)&=&-\int_{\mathbb{P}^n}{G}_{\mu}(\omega+dd^{c}u)^{n}\\
&=&-\sum_{j=0}^n C^j_n\int_{\mathbb {P}^n} {G}_{\mu}(dd^{c}u)^{j}\wedge \omega^{n-j}\\
&=&-\sum_{j=0}^n C^j_n\int_{\mathbb{P}^n} {G}_{\mu}dd^{c}u\wedge(dd^{c}u)^{j-1}\wedge \omega^{n-j}\\
&=& \alpha_n n! \mu(\mathbb P^n)-\sum_{j=1}^n C^j_n\int_{\mathbb{P}^n} {G}_{\mu}dd^{c}u\wedge(dd^{c}u)^{j-1}\wedge \omega^{n-j}.
\end{eqnarray*}
Integrating by parts  thanks to Stokes formula, we infer
\begin{eqnarray*}
I (\mu) &=&\alpha_n n! \int_{\mathbb P^n}\omega^n-\sum_{j=1}^nC^j_n \int_{\mathbb{P}^n}udd^{c} {G}_{\mu}\wedge(dd^{c}u)^{j-1}\wedge \omega^{n-j}\\
&=& \alpha_n n! \int_{\mathbb P^n}\omega^n  -\sum_{j=1}^n C^j_n\int_{\mathbb{P}^n} u (\omega+dd^c {G}_{\mu}-\omega)\wedge(\omega +dd^{c}u-\omega)^{j-1}\wedge \omega^{n-j}\\
&=&\alpha_n n! \int_{\mathbb P^n}\omega^n-\\
&& \sum_{j=1}^n\sum_{k=0}^{j-1}C^k_{j-1} C^j_n(-1)^{j-1-k}\int_{\mathbb{P}^n}u(\omega+dd^c {G}_{\mu}-\omega)\wedge(\omega +dd^{c}u)^k\wedge\omega^{j-1-k}\wedge \omega^{n-j}\\
&=&\alpha_n n! \int_{\mathbb P^n}\omega^n\\
&&-\sum_{j=1}^n\sum_{k=0}^{j-1}C^k_{j-1} C^j_n(-1)^{j-1-k}\int_{\mathbb{P}^n}u(\omega+dd^c {G}_{\mu}-\omega)\wedge(\omega +dd^{c}u)^k\wedge \omega^{n-k-1}\\
&=&\alpha_n n! \int_{\mathbb P^n}\omega^n\\&&+ \sum_{j=1}^n\sum_{k=0}^{j-1}C^k_{j-1} C^j_n(-1)^{j-1-k}\Big[-\int_{\mathbb{P}^n}u(\omega+dd^c {G}_{\mu})\wedge(\omega +dd^{c}u)^k\wedge \omega^{n-k-1}\\
&&\;\;\;\;\;\;\;\;\;\;\;\;\;\;\;\;\;\;\;\;\;\;\;\;\;\;\;\;\;\;\;\;\;\;\;\;\;\;\;\;\;\;\;\;\;\;\;\;\;
+\int_{\mathbb{P}^n}u(\omega +dd^{c}u)^k\wedge \omega^{n-k}\Big]\\
&\leq&\alpha_n n! \int_{\mathbb P^n}\omega^n+\sum_{j=1}^n\sum_{k=0}^{j-1}C^k_{j-1} C^j_n\Big[\int_{\mathbb{P}^n}(\omega+dd^c {G}_{\mu})\wedge(\omega +dd^{c}u)^k\wedge \omega^{n-k-1}\\
&& \;\;\;\;\;\;\;\;\;\;\;\;\;\;\;\;\;\;\;\;\;\;\;\;\;\;\;\;\;\;\;\;
+\int_{\mathbb{P}^n}(\omega +dd^{c}u)^k\wedge \omega^{n-k}\Big].
\end{eqnarray*}
It follows from Chern$-$Levine$-$Nirenberg inequalities for $G_{\mu}\in PSH(\mathbb{P}^n,\omega)\cap L^{1}(\mathbb{P}^n)$ and  $u\in PSH(\mathbb{P}^n,\omega)\cap L^{\infty}(\mathbb{P}^n)$ that the first integral is bounded by a constant depending only on $n$.
\begin{eqnarray*}
\int_{\mathbb{P}^n}(\omega+dd^c G_{\mu})\wedge(\omega +dd^{c}u)^k \wedge \omega^{n-k-1}
&\leq& C_1(n)\parallel G_\mu\parallel_{L^1(\mathbb{P}^n)}\parallel u\parallel^k_{L^\infty(\mathbb{P}^n)}\\
&\leq & C_2(n).
 \end{eqnarray*}
 Also for the last integral we have,
\begin{eqnarray*}
\int_{\mathbb{P}^n}(\omega +dd^{c}u)^k\wedge \omega^{n-k}&\leq& C_3(n)\parallel u\parallel^k_{L^\infty(\mathbb{P}^n)}.
 \end{eqnarray*}
 Then if $- 1 \leq u \leq 0$, we have
$$ 
0 \leq I(\mu)\leq  s_n.
$$
\end{proof} 
 \subsection{Proof of Theorem A}
\begin{proof} We can assume that $ \mathrm{Cap}_{FS} (E) > 0$. Then by \cite{GZ05}, the Monge-Amp\`ere capacity  of $E$ is given by the following formula:
$$
 Cap_{FS} (E):=\sup\Big\{ \int_{E}(\omega+dd^{c}u_{E})^{n} /u\in PSH(\mathbb P^n,\omega),\;\;-1\leq u\leq 0    \Big \}.
$$
Let then $u\in PSH(\mathbb P^n,\omega)$, $-1\leq u\leq 0$, and let   $\mu= (\omega+dd^{c}u)^{n} $. By lemma \ref{lem: bounded}, we have  $I(\mu)\leq s_n,$ where $s _n > 0 $ is an absolute constant. Let $\nu=\frac{\mu}{\sqrt{s_n}}$, then $\nu\in  \mathcal M^+ (\mathbb{P}^n)$ and verifies $ I (\nu) \leq 1$. The proposition \ref{pro:dualité} implies that $$\nu (E) \leq \frac{1}{ \sqrt{- \log \kappa  (E)} }.$$ So
$$
\int_{E}(\omega+dd^{c}u)^{n}\leq \frac{\sqrt{s_n}}{ \sqrt{- \log \kappa  (E)} }.
$$
Since $-1\leq h_E^\ast\leq 0$, we also have 
$$
Cap_{FS}^\ast(E) = \int_{\mathbb P^n}(-h_E^\ast)(\omega+dd^{c}h^{\ast}_{E})^{n}\leq \frac{\sqrt{s_n}}{ \sqrt{- \log \kappa  (E)} }.
$$
 \end{proof}
For the converse of the proposition A is not true. We have the following example which confirms this guess. 
\begin{exa}
 The set  $\mathbb{D}\times \{0\}\subset \mathbb{C}^{2}$ is pluripolar but  non- $\kappa-$polar. Indeed, the set 
 \begin{eqnarray*}
 \mathbb{D}\times \{0\}&=&\{ (z_1,0)\in\mathbb{C}^2/\;\; \vert z_1\vert\leq 1\}\subset \mathbb{C}^2\\
 &\subset & \{ (z_1,z_2)\in\mathbb{C}^2/\;\; \log\vert z_2\vert=-\infty\},
 \end{eqnarray*}
 is pluripolar in $\mathbb{C}^2$. Consider the normalized Lebesgue measure  
 $$\mu= 1_{\bar{\mathbb{D}}}\lambda\otimes [z_2=0],
 $$ 
 on the disk $\mathbb{D}\times \{0\}\subset \mathbb{C}^{2}$. As the support of the measure $\mu$ is contained in $\mathcal U_0$, we have the logarithmic potential $G_\mu$ in affine coordinates(see \cite{AAZ18}): $G_\mu (\zeta) = \mathcal N_\mu (z)$, where
 $$
\mathcal N_{\mu}(z):={1\over 2}\int_{\mathbb{C}^{2}}\log\frac{\vert z-w\vert^{2}+\vert z\wedge w\vert^{2}}{( 1+\vert w\vert^{2})( 1+\vert z\vert^{2} )}d\mu(w),
 $$
 then, 
\begin{eqnarray*}
\mathcal N_{\mu}(z)&=& \frac{1}{2}\int_{\mathbb{D}\times \{0\}}\log\frac{\vert z-w\vert^{2}+\vert z\wedge w\vert^{2}}{( 1+\vert w\vert^{2})( 1+\vert z\vert^{2})}d\lambda(w_1)\times[w_2=0]\\
&=&\frac{1}{2}\int_{\mathbb{D}}\log\frac{\vert z_1-w_1\vert^{2}+\vert z_2\vert+\vert z_2\vert^2\vert w_1\vert^{2}}{( 1+\vert w_1\vert^{2})( 1+\vert z\vert^{2})}d\lambda(w_1).
\end{eqnarray*} 
We calculate $I(\mu)$ as follows:
\begin{eqnarray*}
  I(\mu)&=&-\int_{\mathbb{D}\times \{0\}}\mathcal N_{\mu}(z)d\lambda(z_1)\times[z_2=0]\\
  &=&-\frac{1}{2}\int_{\mathbb{D}}\int_{\mathbb{D}}\log\frac{\vert z_1-w_1\vert^{2}}{ (1+\vert w_1\vert^{2})(1+\vert z_1\vert^{2})}d\lambda(z_1)d\lambda(w_1)\\
  &<&+\infty,
\end{eqnarray*} 
which proves that $\mathbb{D}\times \{0\}$ is non- $\kappa-$polar. 
\end{exa}
Our purpose here is to give an important property of logarithmic capacity.
 \begin{theo}\label{prop: ensemble kappa polaire} Let  $E \subset \mathbb P^n$ be a Borel set. Suppose that there exist $ \mu \in \text{Prob}  (\mathbb P^n)$ such that   $E=\{ G_\mu=-\infty\}$ and $I (\mu) < \infty$.  Then  the set  $E$ is $\kappa-$polar.  
  \end{theo}
\begin{proof}   
 Suppose that   $E=\{ G_\mu=-\infty\}$, where $ \mu \in \text{Prob}  (\mathbb P^n)$ such that $I (\mu) < + \infty$.  Then  $G_\mu$ is $\omega$-plurisubharmonic in $\mathbb P^n$. \\
To show that $E$ is $\kappa$-polar, it is enough   by the  proposition \ref{pro:dualité} to show that for positive Borel measure  $\nu$ such that $I (\nu) < + \infty$, we have $\nu (E) = 0$.\\
 Indeed, put $E_s := \{G_\mu < - s\}$ for $s > 0$. Then for all $s > 0$, we have $- G_\mu \geq s$ on $E_s$ and then 
 $$
 s \nu (E_s) \leq  \int_{\mathbb P^n} (- G_\mu) d \nu.
 $$
 By definition, we have $\int_{\mathbb P^n} (- G_\mu) d \nu = I (\mu,\nu) $. By lemma \ref{lem:Fonda} we have  $ 0 \leq I (\mu,\nu) < + \infty$.
 Hence: $\forall s > 0$,
 $$
   \nu (E_s) \leq  \frac{I (\mu,\nu)}{s}.
 $$
 It follows that $\nu (E) =0$. 
 \end{proof}
 \begin{rem} It would be interesting to establish the theorem  in the case  $I (\mu) = + \infty$.
\end{rem}
In the following  proposition,  we will be interested in the measures that minimizes the energy inspired by the
work of \cite{ST97}.
  \begin{pro}  \label{prop:equilibrium}
  Let   $E \subset \mathbb P^n$ be a closed subset  such that $\kappa(E) > 0$. Then, there exists a  measure $\mu_E \in  \mrm{Prob}  (\mathbb P^n)$ such that $ \mrm{ Supp} \mu_E \subset E $ and $\gamma (E) = I (\mu_E),$ i.e 
  $$
   \kappa (E) = e^{- I (\mu_E)}.
  $$
  \end{pro} 
  \begin{proof} By definition there exists a minimizing sequence  $(\mu_j)$ in   $ \text{Prob}  (\mathbb P^n)$ with $Supp \mu_j \subset E$   such that  
  $$
 \gamma(E) := \lim_{j \to + \infty} I (\mu_j).
  $$
  By compactness of the space  $ \mrm{Prob}  (\mathbb P^n)$, taking a susbsequence if necessary, we can assume that  $\mu_j \to^* \mu$, where $\mu \in  \mrm{Prob}  (\mathbb P^n)$ with  $\mrm{ Supp} \mu \subset E$.\\
  By lower semi-continuity of the energy ( see \ref{pro:lower}), we get 
 $$ I (\mu) \leq \liminf_{j \to + \infty} I (\mu_j) = \gamma (E)$$
 By definition, $\gamma (E) \leq I (\mu)$, this yields the equality $\gamma (E) = I (\mu)$. 
\end{proof} 
The measure $\mu_E $ is called an equilibrium measure for $E$.\\
Now, we prove that the projective logarithmic capacity $\kappa$ defined before is a capacity in the sense of Choquet. 
 We start by recalling the definition of the outer and inner projective logarithmic capacities: 
\begin{defi}
Let $E$ be a subset of  $\mathbb P^n$. We define the outer and inner projective logarithmic capacities of $E$ by: 
$$
\kappa^*(E):=\inf\limits_{\underset{O \textit{ ouvert de $\mathbb P^n$}}{O\supset E}}\kappa(O),
$$
$$
\kappa_*(E):=\sup\limits_{\underset{K \textit{ compact de $\mathbb P^n$}}{K\subset E}}\kappa(K).
$$
The set $E$ is called capacitable if $\kappa^*(E) = \kappa_*(E)$.
\end{defi}
Observe that by the  proposition \ref{pro:dualité}, we have $\kappa_* (E) = \kappa (E)$ when $E$ is Borel.
 \begin{theo} 
The set function  $\kappa^*$ is a capacity on $\mathbb P^n$. In particular, every Borel set is capacitable.
 \end{theo} 
  \begin{proof}
   To prove that   $\kappa^*$ is a capacity on $\mathbb P^n$, we verify first that $\kappa$  defined  on Borel sets $\mathcal B (\mathbb P^n) \longrightarrow \R^+$ is precapacity i.e., it is enough to check the following conditions: 
\begin{enumerate}
\item $\kappa (\emptyset) = 0$ and if $A\subset B$ are the Borels of  $\mathbb P^n$ then $$\kappa(A)\leq \kappa(B).$$
\item If $K_{1}\supset K_{2}\supset K_{3}\supset...$ are compact subsets of $\mathbb P^n$, and  
$K=\displaystyle\cap_{n} K_{n}$ then
$$\kappa(K)=\displaystyle \lim_{n\longrightarrow +\infty}\kappa(K_{n}).$$
\item If $B_{1}\subset B_{2} \subset...$ are Borel subsets of  $\mathbb P^n$, and $B=\cup_{n} B_{n}$, then 
$$\kappa(B)=\lim_{n\longrightarrow +\infty}\kappa(B_{n})=\sup_{n}\kappa(B).$$
 \end{enumerate} 
The proof of  $1)$: we have
  $$
  \kappa (A) := e^{- \gamma (A)}=e^{-\inf \{I (\mu) ; \mu \in \text{Prob}  (\mathbb P^n), \mrm{ Supp} \mu \subset A \}}
 $$ and
 
  $$
  \kappa (B) := e^{- \gamma (B)}=e^{-\inf \{I (\mu) ; \mu \in \text{Prob}  (\mathbb P^n), \mrm{ Supp} \mu \subset B \}}.
 $$ Since $$\gamma (A)\geq \gamma (B) $$ then $$ e^{- \gamma (A)}\leq e^{- \gamma (B)}$$ we obtain
 $$
 \kappa (A)\leq \kappa (B).
 $$
 We now prove  $2)$. Applying $1)$ we get that  $$(\star)\;\;\;  \; \kappa(K_{1})\geq \kappa(K_{2})\geq \kappa(K_{3})\geq...\geq \kappa(K).$$
  For each   $n\geq 1$ let $\nu_{n}$ be an equilibrium measure for  $K_{n}$.  Then  $\nu_{n}\in \text{Prob}  (\mathbb{P}^n)$ for all $n$, by compactness of the space  $ \mrm{Prob}  (\mathbb P^n)$, taking  a subsequence  $(\nu_{n_{k}})$, we  can assume  that  $\nu_{n_{k}} \to^* \nu$, where $\nu \in  \mrm{Prob}  (K_{1})$.  By lower semi-continuity of the energy  (Proposition \ref{pro:lower}), we get 
$$ I (\nu)\leq \liminf_{j \to + \infty} I (\nu_{n_{k}}).$$ 
Moreover, since   $\mrm{ Supp}\nu_{n}\subset K_{n}$ for all  $n$, it follows that  $\mrm{ Supp}\nu\subset K$, and so  $e^{-I(\nu)}\leq \kappa(K)$. thus we obtain 
 $$
 \limsup_{k \to + \infty} \kappa(K_{n_{k}})\leq e^{-I (\nu)}\leq\kappa(K),
 $$
 $$
 \limsup_{k \to + \infty} \kappa(K_{n_{k}})\leq\kappa(K),
 $$
 and combining with $(\star)$ we get the desired conclusion.\\
We now prove $3)$. Using again $1)$, we get 

$$ (\star \star)\;\;\;
\kappa(B_{1})\leq \kappa(B_{2})\leq ...\leq \kappa(B).
$$
let  $K$  be a compact subset of $B $,  and let $\nu$ be an equilibrium measure for  $K$. Since $\nu(B_{n}\cap K)\rightarrow \nu(K)=1$ as $n\rightarrow \infty$, we can produce compact sets  $K_{n}\subset B_{n}\cap K$ such that $K_{1}\subset K_{2}\subset ...$ and  $\nu(K_{n})\rightarrow 1$.  For  $n$ sufficiently large, we have   $\nu(K_{n})>0$, we define 
$$
\mu_{n}=\frac{\nu\vert K_{n}}{\nu(K_{n})}.
$$
Then  $\mu_{n} $ is a Borel Probability measure on $K_{n}$ and  
$$
I(\mu_{n})=\frac{-1}{\nu(K_{n})^{2}}\int_{K}\int_{K} \log \left (\frac{\vert \zeta \wedge \eta\vert}{\vert \zeta\vert \vert \eta\vert}\right) 1_{K_{n}}(\zeta)1_{K_{n}}(\eta) d \nu (\zeta) \ d \nu (\eta),
$$
as  $n\rightarrow \infty$, we have  $\nu(K_{n})\rightarrow 1$ and  $1_{K_{n}}\uparrow 1_{K}$ for $\nu-$almost everywhere, so 
$$
\lim_{n\rightarrow \infty }I(\mu_{n})= -\int_{K}\int_{K} \log \left (\frac{\vert \zeta \wedge \eta\vert}{\vert \zeta\vert \vert \eta\vert}\right)d \nu (\zeta) \ d \nu (\eta)=I(\nu).
$$
Since each  $\mu_{n}$ is supported on a compact subset  $B_{n}$, we have  $\kappa(B_{n})\geq e^{I(\mu_{n})}$, and  it follows that   $$\liminf_{n\rightarrow\infty }\kappa(B_{n})\geq \kappa(K).$$ 
 Finally, as $K$ is an arbitrary subset of $B$, by the   proposition \ref{pro:lower}, we have  
 $$\liminf_{n\rightarrow\infty }\kappa(B_{n})\geq \kappa(B).$$ 
 Using $(\star\star)$ we obtain the desired conclusion.\\
 The previous properties show that the sets function defined on the Borel $\sigma-$Algebra subsets
  $$
  \mathcal B (\mathbb P^n) \ni E \longmapsto \log \frac{1}{ \sqrt{- \log \kappa  (E)} },
  $$ 
 is a precapacity on  $\mathbb P^n$. By the  Proposition \ref{pro:dualité}, it is sub-additive. It follows by Choquet Theory that the set functions 
    $$
     2^{\mathbb P^n} \ni E   \longmapsto \log \frac{1}{ \sqrt{- \log \kappa^*  (E)} },
    $$ 
    is a capacity on  $\mathbb P^n $. Using the same reasoning for $\kappa^*$ as (see \cite{Ch55}, \cite{GZ17}). 
   
  The last assertion is consequence of version Choquet's capacitability theorem (see \cite{Ch55}, \cite{GZ17}).
  \end{proof}
 We will say that a  subset $E \subset \mathbb P^n$  is $\kappa^*$-polar if $\kappa^* (E) = 0$.\\
  We have also the subadditivity property which follows from proposition \ref{pro:dualité}.
   \begin{pro} The set function $E \longmapsto  \frac{1}{ \sqrt{- \log \kappa^*  (E)} }$ is subadditive i.e. if $(E_j)$ is a sequence of subsets of  $\mathbb P^n$ and $E := \cup_j E_j$, then
  $$
  \frac{1}{\sqrt{- \log \kappa^*  (E)} } \leq \sum_{j = 0}^{+ \infty}  \frac{1}{ \sqrt{- \log \kappa^*  (E_j)}}.
 $$
 In particular a countable union of $\kappa^*$-polar sets  is  $\kappa^*$-polar.
\end{pro} 
\section{Logarithmic capacity and Hausdorff measure}
\subsection{ Hausdorff measure}
Let  $f(r)$ be a continuous function defined for  $r\geq 0$ with the properties:\\ 
\begin{center}
 $f(0)=0$,\ \ \ 
 $f(r)$ increasing. 
\end{center}
  Let $E$ be a subset of  $\mathbb{P}^n$, and  $\rho\geq 0$ a real number.  Define
$$
\mathcal{H}^f_\rho(E)=\inf\{ \sum_{i=1}^{+\infty} f(r_i), \;\;\; \cup    \mathcal{S}_i\supset E,\;\;r_i<\rho \},
$$
where the infimum is taken over all countable covers of
 $E$, and  $\mathcal{S}_i$ sphere with radii  $r_i<\rho$.\\
 If we allow   $\rho$ to approach zero, we have 
 
 $$
 \lim_{\rho \longrightarrow 0} \mathcal{H}^f_\rho(E)=\mathcal{H}^f(E),
 $$
 
 we call this {\it limit Hausdorff measure} of  $E$ associated to $f$.\\
  We now give some measures whose potentials have finite energy:
\begin{theo}\label{theo:finie} Let  $ h:[0.1]\longrightarrow \mathbb{R}^{+}$,\;such that\; $h(0)=0$ and

\begin{equation}\label{eq: 1}
 \int_{0}^{1}\frac{h(s)}{s}ds<+\infty.
\end{equation}
If
\begin{equation}\label{eq: 2}
 \mu(B(.,s))\leq h(s),\;\;\textit{where}\;\; B(.,s)\subset \mathbb P^n.
\end{equation} 
Then  $G_{\mu}$ is bounded and $I(\mu)<+\infty$.
 \end{theo}
\begin{proof}
For all  Probability measure satisfying \ref{eq: 2} we have:
\begin{eqnarray*}
 I(\mu):&=&-\int_{\mathbb P^n}\Big(\int_{\mathbb P^n}\log\sin \frac{d(\zeta,\eta)}{\sqrt{2}}d\mu(\eta)\Big)d\mu(\zeta),
 \end{eqnarray*}
  \begin{eqnarray*}
-G_{\mu}(\zeta):&=&-\Big(\int_{\mathbb P^n}\log\sin\frac{d(\zeta,\eta)}{\sqrt{2}}d\mu(\eta)\Big)\\
&=&\int_{0}^{+\infty}\mu(-\log\sin \frac{d(\zeta,\eta)}{\sqrt{2}}>t)dt\\
&=&\int_{0}^{\infty}\mu \left(\{ \eta ; d(\zeta,\eta))<\sqrt{2}\arcsin(e^{-t} )\}\right) dt.
 \end{eqnarray*}
The change of variables: 
 $s=\sqrt{2}\arcsin(e^{-t})$, we have  $ds=\sqrt{2}\frac{-e^{-t}}{\sqrt{1-e^{-2t}}}$  yields   $t=-\log\sin(\frac{s}{\sqrt{2}})$. We conclude that 
  \begin{eqnarray*}
-G_{\mu}(\zeta)&=&-\int_{\frac{\pi}{2}}^{0}\mu(\mathbb{B}(\zeta,s))\coth sds\\
-G_{\mu}(\zeta)&=&\int_{0}^{\frac{\pi}{2}}\mu(\mathbb{B}(\zeta,s))\coth sds\\
&\leq& c.\int_{0}^{\frac{\pi}{2}}h(s)\coth sds,
 \end{eqnarray*}
 so that the right hand side of the above inequality is finite ( this is possible since the function  $\coth s =\frac{1}{s}$ near  0). As  $\int \frac{h(s)}{s}ds<+\infty$, there exists  $C>0$ such that for all  $\zeta\in\mathbb{P}^{n}$, 
 $$
 -C\leq G_{\mu}(\zeta)\leq 0.
 $$
 Which implies that
 $$
 0\leq  I(\mu)\leq C,
 $$ 
 as desired.
 \end{proof}
  Using proposition   \ref{pro: finiteEnergie} and theorem   \ref{theo:finie} we get the following consequence:
 \begin{cor} Let  $0 < m \leq 2 n$ be a fixed real number. Let $\mathcal{H}_m$ denote the Hausdorff measure  (associated to $h$ \ref{eq:h}) of dimension $m$ on  $\C^n \simeq \R^{2n}$. Let $F \subset \mathbb P^n$ be a closed set such that $0 < \mathcal{H}_m (F) < + \infty$. 
Then $\mathcal{H}_m (E) = 0$, for all $\kappa-$polar subset $E \subset F$.
\end{cor} 
\begin{proof}
The measure $\mu := {\bf 1}_F \mathcal{ H}_m$ is Borel measure satisfying the condition  \ref{eq: 1} for an increasing function   $h$   defined by:
\begin{equation}\label{eq:h}
h (s)=s^m,\;\;\; \textit{  for   }  \;\;\;  s>0.
\end{equation}  
By using  theorem \ref{theo:finie}, we conclude that $I(\mu)<+\infty$ and then $\mathcal{H}_m(E)=0$ for all $\kappa-$polar set $E \subset F$, by proposition  \ref{pro: finiteEnergie}.
\end{proof}
\section{Projective transfinite diameter and proof of Theorem B}
   Like in classic setting (\cite{Ran95,ST97,Pa04}), we introduce the associated (projective logarithmic) transfinite diameter   and (projective logarithmic) chebychev constant notions and we connect them to projective logarithmic  capacity $\kappa$.\\
   We recall that
 $$
G(\zeta,\eta):= \log\sin\frac{d(\zeta,\eta)}{\sqrt{2}}\leq 0,\;\;\; \zeta,\;\eta\in \mathbb{P}^n. 
 $$
  For $s \in \N, s\geq 2$, we define the {\it(projective logarithmic) diameter} of order $s \in \N^*$  of $E\subset \mathbb P^n$ by
   $$
 D_s (E) := e^{-\theta_s (E)},
 $$
 where
 \begin{eqnarray*}
\theta_s  &:=&\inf_{\{\zeta_i\in E\}_1^s}\frac{1}{s(s-1)}\sum_{1\leq i\neq j\leq s}(- G(\zeta_i,\zeta_j)).\\
 \end{eqnarray*}
\begin{pro}\label{prop:D_n}  
 The sequence $s \longmapsto D_s := D_s (E)$ is decreasing and then the limit 
 \begin{equation}
 D (E) := \lim_{s\rightarrow+ \infty} D_s (E), 
 \end{equation}
exists   in $\R^+$ and  we have the following
estimate:
$$
D (E) \leq \sin \left(\frac{\mathrm{diam} (E)}{\sqrt{2}}\right),
$$
where $\mathrm{diam} (E)$ is diameter of  $E$ of $\mathbb P^n$ with respect to the Fubini-Study metric.
\end{pro}
\begin{proof} Note that $ \theta_s := - \log D_s (E)$.
Choose $\zeta_1,\cdots,\zeta_s\in E $ arbitrary. If we leave out any index $i=1,\cdots ,s$ then for the remaining $s-1$ points, we obtain by the definition of $\theta_{s-1}$ that
$$
\frac{1}{(s-1)(s-2)}\sum\limits_{\underset{j\neq i,l\neq i}{1\leq j\neq l\leq s}}(-G(\zeta_j,\zeta_l))\geq \theta_{s-1}.
$$
After summing up for  $i=1,2,\cdots,s$ this yields 
$$
\frac{1}{s-1}\sum_{1\leq j\neq l\leq s} (-G(\zeta_i,\zeta_j))\geq s.\theta_{s-1}, 
$$
 for each term $(-G(\zeta_j,\zeta_l))$ occurs exactly  $s-2$  times. Now taking the infimum for all possible  $\zeta_1,\cdots,\zeta_s\in E$, we obtain $s.\theta_s\geq s. \theta_{s-1}$, as claimed. The last property follows from the fact that $D (E) \leq D_2 (E) = \sin \left(\frac{\mathrm{diam} (E)}{\sqrt{2}}\right)$.
\end{proof}
\begin{defi} 
 If $E \subset \mathbb P^n$ is a closed set, then the number
 $$
 D (E) :=\lim_{s\rightarrow +\infty}  D_s (E),
 $$
 is called (projective logarithmic) tranfinite diameter of $E$.
 \end{defi}
 \begin{theo} \label{theo: exists}
 Let    $E \subset \mathbb{P}^n$ be compact set. Then, we have
 $$
  D (E) = \kappa (E).
 $$
 \end{theo}
 The proof of the theorem \ref{theo: exists} relies the following results.
 \begin{lem}\label{lem: D(K)}
Let  $E \subset \mathbb P^n$ be a closed set. Then  
$$ 
- \log D(E)\leq \gamma(E).
$$
\end{lem}
\begin{proof}
Let $\mu\in Prob(E)$ be a arbitrary measure and define $\nu:=\mu^{\otimes s }$ a product measure on  $\mathbb P^n\times\cdots \times \mathbb P^n $. Le us consider the following lower semi-continuous function $h$ defined by
\begin{eqnarray*}
h : (x_1,\cdots, x_s)&\mapsto &  \frac{1}{s(s-1)}\sum_{1\leq i\neq j\leq s }(-G(x_i, x_j)).
\end{eqnarray*}
Since by definition  $0\leq \theta_s \leq h$, it follows that
\begin{eqnarray*}
\theta_s &\leq &\int_{(\mathbb P^n)^s}h (x_1,\cdots, x_s)d\nu(x_1,\cdots,x_s)\\
&=&\int_{(\mathbb P^n)^s}\frac{1}{s(s-1)}\sum_{1\leq i\neq j\leq s }(-G(x_i, x_j))d\nu(x_1,\cdots,x_s)\\
&=& \frac{1}{s(s-1)}\sum_{1\leq i\neq j\leq s } \int_{E^2}(-G(x_i, x_j))d\mu(x_i)d\mu(x_j)=I(\mu).
\end{eqnarray*}
Taking infimum over  $\mu$, yields 
$$
\theta_s (E) \leq \gamma(E). 
$$
Taking the limit in $s$, we obtain
$$
- \log D(E)\leq \gamma(E).
$$
\end{proof}
\begin{lem}\label{lem: gamma}
For all closed set $E \subset \mathbb P^n$, we have 
$$
- \log D(E)\geq \gamma(E).
$$
\end{lem}
\begin{proof}
If $D(E)=0$, then with lemma \ref{lem: D(K)} we have $\gamma(E)=+\infty$, the result holds. Suppose that $D(E) > 0$, and let $\epsilon>0$ fixed. Let $\zeta_1,\cdots,\zeta_s$ be  points of $E$ which realize the infimum of $\theta_s (E)$, called Fekete points. Put $\mu:=\mu_s=\frac{1}{s}\sum_{i=1}^{s}\delta_{\zeta_i}$ where $\delta_{\zeta_i}$ are Dirac measures at the points $\zeta_i,\;\;i=1,\cdots,s$.  For a continue function $h$ such that $0\leq h\leq -G $ and with support in  $E$, we have
\begin{eqnarray*}
\int\int_{E\times E}h d\mu_s d\mu_s&=&\frac{1}{s^2}\sum_{i,j=1}^{s}h(\zeta_i,\zeta_j)\\
&=& \frac{1}{s^2}\sum_{i=1}^{s}h(\zeta_i,\zeta_i)+\frac{1}{s^2}\sum\limits_{\underset{i\neq j}{i,j=1}}^{s}h(\zeta_i,\zeta_j)\\
&\leq& \frac{1}{s^2}\sum_{i=1}^{s}h(\zeta_i,\zeta_i)+\frac{1}{s^2}\sum\limits_{\underset{i\neq j}{i,j=1}}^{s}(-G(\zeta_i,\zeta_j))\\
&\leq& \frac{\Vert h\Vert}{s}+\frac{1}{s^2}\sum\limits_{\underset{i\neq j}{i,j=1}}^{s}(-G(\zeta_i,\zeta_j))\\
&\leq& \frac{\Vert h\Vert}{s}+\frac{s-1}{s}\theta_s\\
&\leq &\frac{\Vert h\Vert}{s} - \log D(E).
\end{eqnarray*}
We have used  proposition \ref{prop:D_n} in last step. In fact, we obtain for $s\geq N=N(\Vert h\Vert, \epsilon) $  the inequality 
\begin{equation}\label{eq:K times K}
\int\int_{E\times E}hd\mu_s d\mu_s\leq - \log D(E)+\epsilon.
\end{equation}
If we take subsequence  $\mu_{s_k}$ of  $\mu$ weakly converges to $\nu$, we have
$$
\int\int_{E\times E}hd\nu d\nu\leq - \log D(E)+\epsilon,
$$
and thus, 
$$
\gamma(E)\leq I(\nu):=\int\int_{E\times E}(-G)d\nu d\nu=\sup\limits_{\underset{h\in C_c(\mathbb P^n\times\mathbb P^n)}{0\leq h\leq -G}}\int\int_{E\times E}hd\nu d\nu\leq - \log  D(K)+\epsilon,
$$
for all $\epsilon> 0$. This shows that $$\gamma(E)\leq - \log D(E). $$
\end{proof}
Finally, combining the lemma  \ref{lem: D(K)} and  \ref{lem: gamma} to conclude  
$$
- \log D(E)=\gamma (E), 
$$ 
for all compact set $E$ of  $\mathbb P^n$,  and the proof of the theorem \ref{theo: exists} is finished.\\
We now have the equidistricution theorem of Fekete points according to equilibrium measure.
\begin{pro}
If  $K\subset \mathbb{P}^n$ is compact set with $\gamma(K)<+\infty$, $\zeta_1,...,\zeta_s\in K$ such that  
 $$
 \theta_s(K)=\frac{-1}{s(s-1)}\sum_{i\neq j}G(\zeta_i,\zeta_j),
 $$
 and  $$
 \nu_s=\sum_{j=1}^s\frac{1}{s}\delta_{\zeta_j},
 $$
 then
 $$
 \nu_s\rightarrow \nu,
 $$
 
weakly in  $\mathbb P^n$. Moreover, $\nu$ is the equilibrium measure of $K$.
 \end{pro}
 \begin{proof}
 Since $\theta_s\rightarrow \gamma(K)$, in the proof of    lemma \ref{lem: gamma} we proved that, if  $\nu$ is the limit of  $\nu_s$ in the weak sense in  $K$, then
 
 $$
 \gamma(K)\leq I(\nu)\leq \liminf_{s\rightarrow +\infty} \theta_s (K) =\gamma(K).
 $$
 Thus, 
 $$
 \gamma(K)=I(\nu),
 $$
 we conclude that  $\nu$ is the equilibrium measure of $K$.
 \end{proof}
 \begin{defi}\label{defi: M_s} Let $K$ be a compact of $\mathbb P^n$. {\it The Chebyshev constant}  of order  $s$ of $K$ is defined by  $\tau_s (K) := - \log M_s (K)$, where 
\begin{eqnarray*}
  M_s (K)  &:=&\sup_{\{\zeta_j\in K\}_1^s}\inf_{\zeta\in K}\frac{1}{s}\sum_{j= 1}^s (-G(\zeta,\zeta_j)).
 \end{eqnarray*}
\end{defi}
\begin{pro}
Let $K\subset \mathbb P^n$ be a compact set, the sequence of Chebyshev constants $ (\tau_s (K))$ converges in  $\R^+$.
\end{pro}
\begin{proof}
The sum of two log-polynomes, $p(w)=\sum_{i=1}^s(-G(w,\zeta_i))$  of  degree $s$ and $q(w)=\sum_{j=1}^t(-G(w,\zeta_j))$ of degree $t$, is also a log-polyn\^ome of degree  $s+t$. Then 
\begin{equation}\label{eq: M_r}
(s+t)M_{s+t}\geq sM_s+tM_t,
\end{equation}
for all $s,t$. If 
$M_s(K)$ is infinite for some $s$, then all terms follow   $M_{s^{\prime}}(K)$, $s^{\prime}\geq s$ are infinite also, thus the convergence.\\
Assume now that  $M_s(K)$ is a finite sequence. Let   $r$, $s$ two integers fixed. Then there exists $l=l(s,t)$ and  $r=r(s,t)$, $0\leq r<t$ the integers non-negatives such that
$ s=l.t+r$. Using  \ref{eq: M_r} we have

$$
s.M_s\geq l(tM_t)+rM_r= sM_t+r(M_r-M_t).
$$
Fixing now a value of $t$, all the values possibles of  $r$ are bounded by $t$.

Divided both sides of the inequality by $s$, and taking  $\liminf_{s\longrightarrow +\infty}$,
$$
\liminf_{s\longrightarrow +\infty}M_s\geq \liminf_{s\longrightarrow +\infty}\Big( M_t+\frac{r}{s} (M_r-M_t) \Big)=M_t.
$$ 
This holds for any fixed $t\in \mathbb N$, so taking    $\limsup_{t\longrightarrow +\infty}$ we obtain 
$$
\liminf_{s\longrightarrow +\infty}M_s\geq \limsup_{t\longrightarrow +\infty}M_t.
$$
Thus, 
$$
\liminf_{s\longrightarrow +\infty}M_s= \limsup_{t\longrightarrow +\infty}M_t,
$$
this gives the result.
\end{proof}
The number
$$\tau (K):=\lim_{s\rightarrow +\infty} \tau_ s (K)$$
is called {\it the Chebyshev constant} of $K$.\\
In the following, we investigate the connection between the transfinite diameter and the Chebyshev constant.
\begin{lem}\label{lem: D_s}
 Let $K\subset \mathbb{P}^n$ be compact set, then for all $s \geq 2$, we have
 $$\theta_{s} (K) \leq M_s.$$
 In particular $\tau (K) \leq D (K)$.
 \end{lem}
 \begin{proof}
We have 
 \begin{eqnarray*}
 \theta_{s+1} &:=&\inf_{\{\zeta_i\in K\}_1^{s+1}}\frac{1}{s(s+1)}\sum_{1\leq i\neq j\leq s+1}(-G(\zeta_i,\zeta_j)).
 \end{eqnarray*}
  Since  $G$ is continuous in $ (\mathbb{P}^n\times\mathbb{P}^n)\setminus\{(\zeta,\zeta),\;\; \zeta\in\mathbb{P}^n \}$, we find $\{\zeta_i\}_1^{s+1}$ such that, 
 \begin{eqnarray*}
 \theta_{s+1}&=&-\frac{1}{s(s+1)}\sum_{1\leq i\neq j\leq s+1}G(\zeta_i,\zeta_j).
 \end{eqnarray*}
 Let $k\in \{1,...,s+1\}$ be given.  We define for $\zeta\in K$ 
 $$
 R_k(\zeta)=\sum\limits_{\underset{j\neq k}{j=1}}^{s+1}(-G(\zeta,\zeta_j)).
 $$
 Since  $ -G(\zeta,\zeta_j)\geq 0,$  we have
 \begin{eqnarray*}
 \min_{\zeta\in K} R_k(\zeta)=R_k(\zeta_k)&=&\inf_{\zeta\in K} \Big( \sum\limits_{\underset{j\neq k}{j=1}}^{s+1}(-G(\zeta,\zeta_j))\Big)\\
 &\leq& sM_s.
 \end{eqnarray*}
Hence
  \begin{eqnarray*}
 \theta_{s+1}&=&\frac{1}{s(s+1)}\sum_{ i=1}^{s+1}\sum_{ i\neq j}(-G(\zeta_i,\zeta_j))\\
 &=&\frac{1}{s(s+1)}\sum_{ i=1}^{s+1}R_i(\zeta_i)\\
 &\leq& \frac{1}{s(s+1)} sM_s(s+1)=M_s.
 \end{eqnarray*}
Since the sequence $\theta_s$ is increasing, we have
 $$
 \theta_s\leq M_s.
 $$
 \end{proof}
\subsection{Proof of Theorem B}  We will now prove the analogous stronger form of the classical Evans's theorem (see \cite{Ran95}).\\
The proof relies on the definition \ref{defi: M_s} and the lemma \ref{lem: D_s}.
 \begin{proof}
By the  definition of $M_s := M_s (E)$, there exists  $\zeta_1,...,\zeta_s\in E$ such that
  $$
 \inf_{\zeta\in E}\frac{1}{s}\sum_{j=1}^s(-G(\zeta,\zeta_j))\geq \frac{1}{2}M_s,
 $$
and let $\mu_s$ be a Probability measure
 $$
 \mu_s=\frac{1}{s}\sum_{j=1}^s\delta_{\zeta_j},\;\;\; \textit{Supp}\mu_s\subset E.
 $$
Since $M_s\rightarrow +\infty,$  there is a sequence  $\{s_h\}$ such that 
\begin{equation}\label{eq: K}
\int_E G(\zeta,\eta)d\mu_{s_h}(\eta)\leq -2^h,\;\;\;\;\; \forall \zeta\in E.
\end{equation}
We define the Probability measure
$$
\mu=\sum_{h=1}^{+\infty}\frac{1}{2^h}\mu_{s_h},\;\;\; \textit{Supp}\mu\subset E.
$$
Then, 
\begin{eqnarray*}
G_\mu(\zeta)&=&\sum_{h=1}^{+\infty}\frac{1}{2^h}G_{\mu_{s_h}}(\zeta),\;\;\;   \zeta\in \mathbb{P}^n\\
&=&-\infty\;\;\;\textit{for all}\;\; \zeta \in E.
\end{eqnarray*}
We show now that
$$
 \forall \zeta\in \mathbb{P}^n\setminus E,\;\;\; G_\mu(\zeta)>-\infty.
 $$
Indeed, the classical inequality  
 $
 \frac{2}{\pi} t \leq \sin t \leq t,\;\;\; t \in [0,\frac{\pi}{2}],$ 
  applied to the kernel $G$ implies 
 \begin{eqnarray*}
 G(\zeta,\eta)&=&\log\sin\frac{d(\zeta,\eta)}{\sqrt{2}}\;\;\;\textit{for all }\;\; \eta\in E \\
 &\geq&\log\frac{2d(\zeta,\eta)}{\pi\sqrt{2}}\;\;\;\textit{for all }\;\; \eta\in E\\
 &\geq&\log d(\zeta,\eta)+ C \;\;\;\textit{for all }\;\; \eta\in E\\
 &\geq& \log d(\zeta,E)+ C. 
 \end{eqnarray*}
Then  
$$
G_\mu (\zeta)\geq  \log d(\zeta,E) + C 
>-\infty  \;\;\; \textit{for all }\;\;\zeta\in \mathbb P^n \setminus E.
$$
It follows that $G_{\mu}\not\equiv -\infty$ since $E \neq \mathbb{P}^n$ and that
$$
 \{ \zeta \in \mathbb{P}^n ; G_\mu (\zeta) = - \infty\} \subset E.
$$ 
On the other hand, by the inequality (\ref{eq: K}) we have
 $$
 E \subset \{\zeta\in \mathbb{P}^n,\;\;\; {G}_\mu(\zeta)=-\infty\}.
 $$
This proves the theorem.
 \end{proof}
 \bigskip

\noindent{ \bf Aknowlegements :} 
{\it 
This paper is a part of  the Phd thesis of the second author supervised by the first author and  Ahmed Zeriahi.  This work was completed when  she  was visiting the $''$Institut de Math\'ematiques de Toulouse $''$ in June 2016. She would like to think this institution for the invitation and the Professeur Vincent Guedj for useful discussions and suggestions.

\smallskip

\smallskip

}


 \end{document}